\documentclass[12pt,a4paper]{article}

\usepackage[T1]{fontenc}
\usepackage{ae}	
\usepackage[italian,UKenglish]{babel}    
\usepackage{amsmath}
\usepackage{amsfonts}
\usepackage{amssymb}
\usepackage{amsthm}
\usepackage{algpseudocode}
\usepackage{float}
\usepackage{xcolor}
\usepackage{mathrsfs}
\usepackage{import}
\usepackage{geometry}

\usepackage{lettrine}   

\usepackage{verbatim}
\usepackage{alltt}

\usepackage{graphics}
\usepackage{graphicx}

\graphicspath{{./imgs/}}

\usepackage{subfigure}
\usepackage{floatflt}

\usepackage[font=small,labelfont=bf]{caption}

\theoremstyle{plain}
\newtheorem{theorem}{Theorem}[section]

\newtheorem{definition}[theorem]{Definition}

\newtheorem{lemma}{Lemma}[section]

\newtheorem{proposition}[theorem]{Proposition}
\newtheorem{remark}[theorem]{Remark}

\numberwithin{equation}{section}
\theoremstyle{definition}

\newcommand{\R}{\ensuremath{\mathbb{R}}}

\begin{document}

\title{Blow-up and global existence for solutions to the porous medium equation with reaction and \\ fast decaying density}
%
%
\author{Giulia Meglioli\thanks{Dipartimento di Matematica, Politecnico di Milano, Italia (giulia.meglioli@polimi.it).}\,\, and Fabio Punzo\thanks{Dipartimento di Matematica, Politecnico di Milano, Italia (fabio.punzo@polimi.it).}}
\date{}
%

%

\maketitle              

\begin{abstract}
We are concerned with nonnegative solutions to the Cauchy problem for the porous medium equation with a variable density $\rho(x)$ and a power-like reaction term $u^p$ with $p>1$. The density decays {\it fast} at infinity, in the sense that $\rho(x)\sim |x|^{-q}$ as $|x|\to +\infty$ with $q\ge 2.$
In the case when $q=2$, if $p$ is bigger than $m$, we show that, for large enough initial data, solutions blow-up in finite time and for small initial datum, solutions globally exist.
On the other hand, in the case when $q>2$, we show that existence of global in time solutions always prevails. The case of {\it slowly} decaying density at infinity, i.e. $q\in [0,2)$, is examined in \cite{MP1}.
\end{abstract}

\bigskip
\bigskip



\section{Introduction}
We investigate global existence and blow-up of nonnegative solutions to problem
\begin{equation}
\begin{cases}
\rho(x) u_t = \Delta(u^m)+ \rho(x) u^p \qquad \text{in } \R^N \times (0,\tau) \\
u(x,0)=u_0(x) \qquad \qquad \text{in } \R^N \times \{0\}
\end{cases}
\label{problema}
\end{equation}
where $N\geq 3, u_0 \in L^{\infty}(\R^N), u_0\geq 0$, $\rho \in C(\R^N)$, $\rho > 0$, $p>1, m > 1$ and $\tau>0$. We always assume that
\begin{equation}\tag{{\it $H$}} \label{hp}
\begin{cases}
\text{(i)} \; \rho\in C(\mathbb R^N),\, \rho>0\,\, \textrm{in}\,\, \mathbb R^N\,,\\
\text{(ii)}\,  u_0\in L^\infty(\mathbb R^N), u_0\geq 0\,\, \textrm{in}\,\, \mathbb R^N\,,
\end{cases}
\end{equation}
and that
\begin{equation}\label{hpSup}
\begin{aligned}
&\text{there exist} \,\, k_1, k_2\in (0, +\infty)\,\, \text{with}\,\, k_1\leq k_2, r_0 > 0\,, q\ge 2 \,\, \text{such that} \\
&\,\,\,\,\,\,k_1(|x|+r_0)^q \le\dfrac{1}{\rho(x)}\le k_2(|x|+r_0)^q\quad \text{for all}\,\,\, x\in \R^N\,.
\end{aligned}
\end{equation}
The parabolic equation in problem \eqref{problema} is of the {\it porous medium} type, with a variable density $\rho(x)$ and a reaction term $\rho(x) u^p$.
Clearly, such parabolic equation is degenerate, since $m>1$. Moreover, the differential equation in \eqref{problema} is equivalent to 
\[ u_t =\frac 1{\rho(x)} \Delta(u^m)+ u^p \qquad \text{in } \R^N \times (0,\tau);\]
thus  the related diffusion operator is $\frac 1{\rho(x)}\Delta$, and in view of \eqref{hpSup}, the coefficient $\dfrac1{\rho(x)}$ can positively diverge at infinity.
The differential equation in \eqref{problema}, posed in the interval $(-1,1)$ with homogeneous Dirichlet boundary conditions, has been introduced in \cite{KR3} as a mathematical model of evolution of plasma temperature, where $u$ is the temperature, $\rho(x)$ is the particle density, $\rho(x) u^p$ represents the volumetric heating of plasma. Indeed, in \cite[Introduction]{KR3} a more general source term of the type $A(x) u^p$ has also been considered; however, then the authors assume that  $A\equiv 0$; only some remarks for the case $A(x)=\rho(x)$ are made in \cite[Section 4]{KR3}. Then in \cite{KR1} and \cite{KR2} the Cauchy problem \eqref{problema} is dealt with in the case without the reaction term $\rho(x) u^p.$

\smallskip

In view of \eqref{hpSup} the density $\rho$ decays at infinity. Indeed,
\begin{equation}\label{eqf3}
\frac{1}{k_2(|x|+r_0)^q}\leq \rho(x)\leq \frac{1}{k_1(|x|+r_0)^q}\quad \textrm{for all }\,\,|x|>1\,,
\end{equation}
with
$$q\ge2.$$

\smallskip

Since we assume \eqref{hpSup} with $q\geq 2$,
we refer to $\rho(x)$ as a {\em fast decaying density} at infinity. On the other hand, in \cite{MP1} it is studied problem \eqref{problema} with a {\em slowly decaying density,} that is \eqref{hpSup} is assumed with $q<2$.

There is a huge literature concerning various problems related to \eqref{problema}. For instance, problem \eqref{problema} with $\rho\equiv 1, m=1$ is studied in \cite{CFG, DL, FI, F, H, I, IM, L, Q, Sacks, S, Y}, problem \eqref{problema} without the reaction term $u^p$ is treated in \cite{Eid90, EK, GMPor, GMPfra1, GMPfra2, KKT, KPT, KP1, KP2, KRV10, KR1, KR2, KR3, NR, RA, RV06, RV08, RV09}. Moreover, problem \eqref{problema} with $m=1$ is addressed in \cite{LX} (see also \cite{dPRS}), where $\rho$ satisfies \eqref{eqf3} with $0\leq q< 2$. 
In particular, let us recall some results established in \cite{SGKM} for problem \eqref{problema} with $\rho\equiv 1, m>1, p>1$ (see also \cite{GV, MQV}). We have:
\begin{itemize}
\item (\cite[Theorem 1, p. 216]{SGKM}) For any $p>1$, for all sufficiently large initial data, solutions blow-up in finite time;
\item (\cite[Theorem 2, p. 217]{SGKM}) if $p\in\left(1,m+\frac2N\right)$, for \it all \rm initial data, solutions blow-up in finite time;
\item (\cite[Theorem 3, p. 220]{SGKM}) if $p>m+\frac2N$, for all sufficiently small initial data with compact support, solutions exist globally in time and belong to $L^\infty(\mathbb R^N \times (0, +\infty))$.
\end{itemize}
Similar results for quasilinear parabolic equations, also involving $p$-Laplace type operators or double-nonlinear operators, have been stated in \cite{AfT}, \cite{A1},  \cite{AT1}, \cite{dP1}, \cite{dPS1}, \cite{dPS2},  \cite{IS1}, \cite{IS2}, \cite{MT}, \cite{MTS}, \cite{MTS2},  \cite{MP}, \cite{MP2}, \cite{PT}, \cite{T1}, \cite{WZ} (see also \cite{MMP} for the case of Riemannian manifolds); moreover, in \cite{GMPhyp} the same problem on Cartan-Hadamard manifolds has been investigated.
In particular,  
in \cite[Theorem 2]{MT} it is shown that if $\rho(x)=(1+|x|)^{-q}$ with $0<q<2$, $p>m$, and $u_0$ is small enough (in an appropriate sense), then there exists a global solution; moreover, a smoothing estimate is given. Such result will be compared below with one of our results (see Remark \ref{ossMT}).

\smallskip

In \cite{MP1} the following results for problem \eqref{problema} are established, assuming \eqref{hpSup} with $0\leq q<2$. 
\begin{itemize}
\item (\cite[Theorem 2.1]{MP1}. If
\[p>\bar p,\]
$u_0$ has compact support and is small enough, then there exist global in time solutions to problem \eqref{problema} which belong to $L^\infty(\mathbb R^N \times (0, +\infty))$;
here $\bar p$ is a certain exponent, which depends on $N, m, q, k_1, k_2$. In particular, for $k_1=k_2$ we have
\[\bar p=m+\frac{2-q}{N-q}\,.\]
\item (\cite[Theorem 2.3]{MP1}). For any $p>1$, if $u_0$ is sufficiently large, then solutions to problem \eqref{problema} blow-up in finite time.
\item (\cite[Corollary 2.4]{MP1}). If $1<p<m$, then for any $u_0\not\equiv 0$, solutions to problem \eqref{problema} blow-up in finite time. Moreover (\cite[Theorem 2.5]{MP1}), if $m\leq p<\underline p$, where $\underline p<\bar p$ is a certain exponent depending on $N, m, q, k_1, k_2$, then solutions to problem \eqref{problema} blow-up in finite time for any nontrivial initial datum. For $k_1=k_2$, $\underline p=\bar p.$
\end{itemize}
Analogous results, proved by different methods, can be found also in \cite{MT, MTS}, where also more general double-nonlinear operators are treated.

\subsection{Outline of our results} Let us now describe our main results.
We distinguish between two cases: $q=2$ and $q>2$. First, assume that \eqref{hpSup} holds with $q=2$.

\begin{itemize}
\item (Theorem \ref{teosupersolutioncritical}). If
\[p>m,\]
$u_0$ has compact support and is small enough, then there exist global in time solutions to problem \eqref{problema}, which belong to $L^\infty(\mathbb R^N \times (0, +\infty))$;
\item (Theorem \ref{teosubsolutioncritical}). For any $p>m$, if $u_0$ is sufficiently large, then solutions to problem \eqref{problema} blow-up in finite time.
\end{itemize}
The proofs mainly relies on suitable comparison principles and properly constructed sub- and supersolutions, 
which crucially depend on the behavior at infinity of the inhomogeneity term $\rho(x)$. More precisely,  they are of the type
\begin{equation}\label{e4f}
w(x,t)=C\zeta(t)\left [ 1- \frac{\log(|x|+r_0)}{a} \eta(t) \right ]_{+}^{\frac{1}{m-1}}\quad \textrm{for any}\,\,\, (x,t)\in \big[\mathbb R^N\setminus B_1(0)\big]\times [0, T),
\end{equation}
for suitable functions $\zeta=\zeta(t), \eta=\eta(t)$ and constants $C>0, a>0$. The presence of $\log(|x|+r_0)$ in $w$ is strictly related to the assumption that $q=2$. Observe that the barriers used in \cite{MP1} for the case $0\leq q<2$, which are of power type in $|x|$, do not work in the present situation. Furthermore, note that the exponent $\bar p$ introduced in \cite{MP1} for $0\leq q<2$, when $q=2$ becomes $\bar p=m$. Hence Theorem  \ref{teosupersolutioncritical} can be seen as a generalization of \cite[Theorem 2.1]{MP1} to the case $q=2$.

Now, assume that $q>2$. We have the following results (see Theorem \ref{teosupersolutionsuper} and Remark \ref{oss1}).
\begin{itemize}
\item Let $1<p<m$. Then for suitable $u_0\in L^\infty(\R^N)$ there exist global in time solutions to problem \eqref{problema}. We do not assume that $u_0$ has compact support, but we need that it fulfills a decay condition as $|x|\to+\infty$. However, $u_0$ in a compact subset of $\R^N$ can be arbitrarily large. We cannot deduce that the corresponding solution belongs to  $L^\infty(\mathbb R^N \times (0, +\infty))$, but it is in $L^\infty(\mathbb R^N \times (0, \tau))$ for each $\tau>0.$
\item Let $p>m\geq 1$. Then for suitable $u_0\in L^\infty(\R^N)$, problem \eqref{problema} admits a solution in $L^\infty(\mathbb R^N\times (0, +\infty))$. We need that
\[0\leq u_0(x)\leq C W(x)\,\quad \textrm{for all }\,\, x\in \mathbb R^N,\]
where $C>0$ is small enough and $W(x)$ is a suitable function, which vanishes as $|x|\to +\infty$.
We should mention that, as recalled above, a similar result was been obtained in \cite[Theorem 2]{MT}, where also double-non linear operators are treated; see Remark \ref{ossMT} below.

\item Let $p=m> 1$. Then for suitable $u_0\in L^\infty(\R^N)$, problem \eqref{problema} admits a solution in $L^\infty(\mathbb R^N\times (0, +\infty))$, provided that $r_0>0$ in \eqref{hpSup} is big enough.
\end{itemize}
Such results are very different with respect to the cases $0\leq q<2$ and $q=2$. In fact, we do not have finite-time blow-up, but global existence prevails, for suitable initial data. The results follow by comparison principles, once we have constructed appropriate supersolutions, that have the form
\[w(x,t)=\zeta(t)W(x)\quad \textrm{for all }\,\, (x,t)\in \mathbb R^N\times (0, +\infty),\]
for suitable $\zeta(t)$ and $W(x)$. When $p\geq m$, $\zeta(t)\equiv 1$. Observe that we can also include the linear case $m=1$, whenever $p>m$. In this respect, our result complement the results in \cite{LX}, where only the case $q<2$ is addressed.  Finally, let us mention that it remains to be understood whether in the case $1<p<m$ solutions can blow-up in infinite time or not.

\section{Statements of the main results}

For any $x_0\in \mathbb R^N$ and $R>0$ we set
$$B_R(x_0)=\{x\in \R^N :  \| x-x_0 \| < R \}.$$
When $x_0=0$, we write $B_R\equiv B_R(0).$

\smallskip

For the sake of simplicity, sometimes  instead of  \eqref{hpSub}, we suppose that
\begin{equation}\label{hpSub}
\begin{aligned}
&\text{there exist} \,\, k_1, k_2\in (0, +\infty)\,\, \text{with}\,\, k_1\leq k_2\,, q\ge 2 \,\,, R>0\,\, \text{such that} \,\,\,\,\,\,\,\,\,\,\,\,\,\,\\
&\,\,\,\,\,\,k_1|x|^q \le\dfrac{1}{\rho(x)}\le k_2|x|^q\quad \text{for all}\,\,\, x\in \R^N \setminus B_{R}\,.
\end{aligned}
\end{equation}
In view of \eqref{hp}-(i),
\begin{equation}\label{rho}
\begin{aligned}
&\textrm{for any}\,\, R>0\, \text{there exist} \,\, \rho_1(R), \rho_2(R)\in (0, +\infty)\,\, \text{with}\,\, \rho_1(R)\leq \rho_2(R)\\& \,\,\,\,\,\, \text{such that}
\, \rho_1(R)\leq \frac1{\rho(x)}\leq \rho_2(R) \quad \textrm{for all}\,\,\, x\in \overline{B_{R}}\,.
\end{aligned}
\end{equation}
Obviously, \eqref{hpSup} is equivalent to \eqref{hpSub} and \eqref{rho}.

\medskip

In the sequel we shall refer to $q$ as the order of decaying of $\rho(x)$ as $|x|\to +\infty$.

\subsection{Order of decaying: $q=2$}
Let $q=2$. The first result concerns the global existence of solutions to problem \eqref{problema} for $p>m$. We assume that
\begin{equation}\label{hpC}
r_0 > e, \quad \frac{k_2}{k_1}<(N-2)(m-1)\frac{p-m}{p-1}\log r_0\,.
\end{equation}

\begin{theorem}\label{teosupersolutioncritical}
Assume \eqref{hp}, \eqref{hpSup} for $q=2$ and \eqref{hpC}. Suppose that $$p>m\,,$$
and that $u_0$ is small enough and has compact support. Then problem \eqref{problema} admits a global solution $u\in L^\infty(\mathbb R^N\times (0, +\infty))$. \newline
More precisely, if $C>0$ is small enough, $a>0$ is so that
$$
0<\omega_0\le\frac{C^{m-1}}{a}\le\omega_1
$$
for suitable $0<\omega_0<\omega_1$, $T>0$,

\begin{equation}\label{eq20}
u_0(x) \le CT^{-\frac{1}{p-1}} \left [ 1- \frac{\log(|x|+r_0)}{a} \, T^{-\frac{p-m}{p-1}} \right ]_{+}^{\frac{1}{m-1}} \quad \text{for any}\,\, x\in \R^N\,,
\end{equation}
then problem \eqref{problema} admits a global solution $u\in L^{\infty}(\R^N\times (0,+\infty))$. Moreover,
\begin{equation}\label{eq21}
u(x,t) \le C(T+t)^{-\frac{1}{p-1}} \left [ 1- \frac{\log(|x|+r_0)}{a} \,(T+t)^{-\frac{p-m}{p-1}} \right ]_{+}^{\frac{1}{m-1}} \, \text{for any}\,\, (x,t)\in \R^N \times (0,+\infty)\,.
\end{equation}
\end{theorem}

Observe that if $u_0$ satisfies \eqref{eq20}, then
\begin{equation*}
\operatorname{supp}u_0\subseteq \{x\in \mathbb R^N\,:\, \log(|x|+r_0)\leq a T^{\frac{p-m}{p-1}}\}\,.
\end{equation*}
From \eqref{eq21} we can infer that
\begin{equation}\label{eq26}
\operatorname{supp}u(\cdot, t)\subseteq \{x\in \mathbb R^N\,:\, \log(|x|+r_0)\leq a (T+t)^{\frac{p-m}{p-1}}\}\quad \textrm{for all } t>0\,.
\end{equation}

\medskip

The choice of the parameters $C>0, T>0$ and $a>0$ is discussed in Remark \ref{thmSuperC}.

\bigskip



The next result concerns the blow-up of solutions in finite time, for every $p>m>1$, provided that the initial datum is sufficiently large. We assume that hypothesis \eqref{hpSub} holds with the choice
\begin{equation}\label{hpCsub}
q=2\,, \quad R = e\,.
\end{equation}
So we fix, in assumption \eqref{rho},
$$\rho_1(R)=\rho_1(e)=:\rho_1\,, \quad \rho_2(R)=\rho_2(e)=:\rho_2\,.$$

Let
\[\mathfrak{s}(x):=\begin{cases}
\log(|x|)  &\quad \text{if}\quad  x\in \R^N\setminus B_e, \\
& \\
\dfrac{|x|^2+e^2}{2e^2} &\quad\text{if}\quad  x\in B_e\,.
\end{cases}\]


\begin{theorem}\label{teosubsolutioncritical}

Let assumption \eqref{hp}, \eqref{hpSub} and \eqref{hpCsub}. For any $$p>m$$ and for any $T>0$, if the initial datum $u_0$ is large enough, then the solution $u$ of problem \eqref{problema} blows-up in a finite time $S\in (0,T]$, in the sense that
\begin{equation}\label{eq22}
\|u(t)\|_{\infty} \to \infty \text{ as } t \to S^{-}\,.
\end{equation}
More precisely, if $C>0$ and $a>0$ are large enough, $T>0$,
\begin{equation}\label{eq23}
u_0(x)\ge CT^{-\frac{1}{p-1}}\left[1-\frac{\mathfrak{s}(x)}{a}\,T^{\frac{m-p}{p-1}}\right]^{\frac{1}{m-1}}_{+}\,, \quad \text{for any}\,\, x\in \R^N\,,
\end{equation}
then the solution $u$ of problem \eqref{problema} blows-up and satisfies the bound from below
\begin{equation} \label{eq24}
u(x,t) \ge C (T-t)^{-\frac{1}{p-1}}\left [1- \frac{\mathfrak{s}(x)}{a}\, (T-t)^{\frac{m-p}{p-1}} \right ]_{+}^{\frac{1}{m-1}}, \,\, \text{for any}\,\, (x,t) \in \R^N\times(0,S)\,.
\end{equation}
\end{theorem}

\medskip
Observe that if $u_0$ satisfies \eqref{eq23}, then
\begin{equation*}
\operatorname{supp}u_0\supseteq \{x\in \mathbb R^N\,:\, \mathfrak{s}(x)< a T^{\frac{p-m}{p-1}}\}\,.
\end{equation*}
From \eqref{eq24} we can infer that
\begin{equation}\label{eq25}
\operatorname{supp}u(\cdot, t)\supseteq \{x\in \mathbb R^N\,:\, \mathfrak{s}(x) < a (T-t)^{\frac{p-m}{p-1}}\}\quad \textrm{for all } t\in[0,S)\,.
\end{equation}
\medskip
The choice of the parameters $C>0, T>0$ and $a>0$ is discussed in Remark \ref{thmSubC}.

\subsection{Order of decaying: $q>2$}
Let $q>2$. The first result concerns the global existence of solutions to problem \eqref{problema} for any $p>1$ and $m>1$, $p\neq m$. Let us introduce the parameter $\bar b\in \R$ such that
\begin{equation}\label{hpS}
0<\bar b<\min \{N-2\,,\,\,q-2\}\,.
\end{equation}
Moreover, we can find $\bar c>0$ such that
\begin{equation}\label{eq115}
(r+r_0)^{-\frac{\bar bp}{m}}\le \bar c \quad \text{for any}\,\, r\ge 0\,,
\end{equation}
with $r_0>0$ as in hypothesis \eqref{hpSup}.

\begin{theorem}\label{teosupersolutionsuper}
Let assumptions \eqref{hp}, \eqref{hpSup} and \eqref{hpS} be satisfied with $q>2$. Suppose that $$1<p<m\,,\quad \text{or}\,\,\, p> m\geq 1\,,$$
and that $u_0$ is small enough. Then problem \eqref{problema} admits a global solution $u\in L^\infty(\mathbb R^N\times (0, \tau))$ for any $\tau>0$.
More precisely, we have the following cases.
\begin{itemize}
\item[(a)]\, Let $1<p<m$. If $C>0$ is big enough, $r_0>0$, $T>1$, $\alpha> 0$,
\begin{equation}\label{eq100}
u_0(x) \le CT^{\alpha} \left (|x|+r_0 \right )^{-\frac{\bar b}{m}} \quad \text{for any}\,\, x\in \R^N\,,
\end{equation}
then problem \eqref{problema} admits a global solution $u$, which satisfies the bound from above
\begin{equation}\label{eq101}
u(x,t) \le C(T+t)^{\alpha} \left (|x|+r_0 \right )^{-\frac{\bar b}{m}} \, \text{for any}\,\, (x,t)\in \R^N \times (0,+\infty)\,.
\end{equation}
\item[(b)]\, Let $p> m \geq  1$. If $C>0$ is small enough, $r_0>0$, $T>0$ and \eqref{eq100} holds with $\alpha=0$, then problem \eqref{problema} admits a global solution $u\in L^{\infty}(\R^N\times (0,+\infty))$, which satisfies the bound from above \eqref{eq101} with $\alpha=0$.
\end{itemize}
\end{theorem}

\begin{remark}\label{oss1}
Observe that, in the case when $p=m$, if $C>0$ is small enough, $r_0>0$ big enough to have
$$
\left (\frac{1}{r_0} \right )^{\frac{\bar bp}{m}} \le \bar bk_1(N-2-\bar b)\,,
$$
$T>0$ and \eqref{eq100} holds with $\alpha=0$, then problem \eqref{problema} admits a global solution $u\in L^{\infty}(\R^N\times (0,+\infty))$ which satisfies the bound from above \eqref{eq101} for $\alpha=0$.
\end{remark}

\medskip

Note that in Theorem \ref{teosupersolutionsuper} we do not require that $\operatorname{supp}\, u_0$ is compact.
\medskip

The choice of the parameters $C>0, T>0$ and $a>0$ is discussed in Remark \ref{thmSuperS}.

\bigskip

\begin{remark}\label{ossMT} The statement in Theorem \ref{teosupersolutionsuper}-(b) is in agreement with \cite[Theorem 2]{MT}, where it is assumed that $p>m$, $\rho(x)=(1+|x|)^{-q}$ with $q>2$, 
$\int_{\mathbb R^N} \rho(x) u_0(x) dx <+\infty, \int_{\mathbb R^N} \rho(x)[u_0(x)]^{\bar q} dx<\delta$, for some $\delta>0$ small enough and $\bar q> \frac N 2 (p-m)$. 

Note that the assumption on $u_0$ is of a different type. In particular, in view of \eqref{eq100} and \eqref{hpS}, the initial datum $u_0$ considered in Theorem \ref{teosupersolutionsuper}-(b) not necessarily satisfies 
$\int_{\mathbb R^N} \rho(x) u_0(x) dx <+\infty.$

 In \cite{MT} the proofs are based on the energy method, so they are completely different with respect to our approach.
\end{remark}

\section{Preliminaries}

In this section we give the precise definitions of solutions of all problems we address, then we state a local in time existence result for problem \eqref{problema}. Moreover, we recall some useful comparison principles. The proofs of such auxiliary results can be found in \cite[Section 3]{MP1}.

\medskip

Throughout the paper we deal with {\em very weak} solutions to problem \eqref{problema} and to the same problem set in different domains, according to the following definitions.

\begin{definition}
Let $u_0\in L^{\infty}(\R^N)$ with $u_0\ge0$. Let $\tau>0$, $p>1, m>1$. We say that a nonnegative function $u\in L^{\infty}(\R^N\times (0,S))$ for any $ S<\tau$ is a solution of problem \eqref{problema} if
\begin{equation}
\begin{aligned}
-\int_{\R^N}^{}\int_{0}^{\tau} \rho(x) u \varphi_t \,dt\,dx &= \int_{\R^N} \rho(x) u_0(x) \varphi(x,0) \,dx \\ &+ \int_{\R^N}^{}\int_{0}^{\tau}  u^m \Delta \varphi \,dt\,dx \\ &+ \int_{\R^N}^{}\int_{0}^{\tau} \rho(x) u^p \varphi \,dt\,dx
\end{aligned}
\label{veryweak}
\end{equation}
for any $\varphi \in C_c^{\infty}(\R^N \times [0,\tau)), \varphi \ge 0.$ Moreover, we say that a nonnegative function $u\in L^{\infty}(\R^N\times (0,S))$ for any $ S<\tau$ is a subsolution (supersolution) if it satisfies \eqref{veryweak} with the inequality $"\le"$ ($"\ge"$) instead of $"="$ with $\varphi\geq 0$.
\label{soluzioneveryweak}
\end{definition}

For every $ R>0$, we consider the auxiliary problem
\begin{equation}
\begin{cases}
 u_t=\frac 1{\rho(x)}\Delta(u^m) +u^p & \text{ in } B_R\times(0,\tau) \\
u=0  &  \text{ on } \partial B_R \times(0,\tau)\\
u=u_0 &  \text{ in } B_R\times\{0\}\,.\\
\end{cases}
\label{problemalocale}
\end{equation}

\begin{definition}
Let $u_0\in L^{\infty}(B_R)$ with $u_0\ge0$. Let $\tau>0$, $p>1, m>1$. We say that a nonnegative function $u\in L^{\infty}(B_R\times (0,S))$ for any $S<\tau$ is a solution of problem \eqref{problemalocale} if
\begin{equation}
\begin{aligned}
-\int_{B_R}^{}\int_{0}^{\tau} \rho(x) u\, \varphi_t \,dt\,dx &= \int_{B_R} \rho(x) u_0(x) \varphi(x,0) \,dx \\ &+ \int_{B_R}^{}\int_{0}^{\tau} u^m \Delta \varphi \,dt\,dx \\ &+ \int_{B_R}^{}\int_{0}^{\tau} \rho(x) u^p \varphi \,dt\,dx
\end{aligned}
\label{veryweaklocale}
\end{equation}
for any $\varphi \in C_c^{\infty}(\overline{B_R} \times [0,\tau))$ with $\varphi| _{\partial B_R}=0$ for all $t\in [0,\tau)$. Moreover, we say that a nonnegative function $u\in L^{\infty}(B_R\times (0,S))$ for any $S<\tau$ is a subsolution (supersolution) if it satisfies \eqref{veryweaklocale} with the inequality $"\le"$ ($"\ge"$) instead of $"="$, with $\varphi\geq 0$.
\label{soluzioneveryweaklocale}
\end{definition}

\begin{proposition}\label{exiloc}
Let hypothesis \eqref{hp} be satisfied. Then there exists a solution $u$ to problem \eqref{problemalocale} with
$$\tau\geq \tau_R:=\frac{1}{(p-1)\|u_0\|_{L^\infty(B_R)}^{p-1}}.$$
\end{proposition}

Moreover, the following comparison principle for problem \eqref{problemalocale} holds (see \cite{ACP} for the proof).
\begin{proposition}
Let assumption \eqref{hp} hold.
If $u$ is a subsolution of problem \eqref{problemalocale} and $v$ is a supersolution of \eqref{problemalocale}, then
$$u\le v \quad  \textrm{a.e. in } \, B_R \times (0,\tau).$$
\label{confrontolocale}
\end{proposition}

\begin{proposition}
Let hypothesis \eqref{hp} be satisfied. Then there exists a solution $u$ to problem \eqref{problema} with
$$\tau\geq \tau_0:=\frac{1}{(p-1)\|u_0\|_{\infty}^{p-1}}.$$
Moreover, $u$ is the {\em minimal solution}, in the sense that for any solution $v$ to problem \eqref{problema} there holds
\[u\leq v \quad \textrm{in }\,\,\, \mathbb R^N\times (0, \tau)\,.\]
\label{prop1}
\end{proposition}

In conclusion, we can state the following two comparison results, which will be used in the sequel.

\begin{proposition}\label{cpsup}
Let hypothesis \eqref{hp} be satisfied. Let $\bar{u}$ be a supersolution to problem \eqref{problema}. Then, if $u$ is the minimal solution to problem \eqref{problema} given by Proposition \ref{prop1}, then
\begin{equation}\label{eq172}
u\le\bar{u} \quad \text{a.e. in } \R^N \times (0,\tau)\,.
\end{equation}
In particular, if $\bar{u}$ exists until time $\tau$, then also $u$ exists at least until time $\tau$.
\end{proposition}

\begin{proposition}\label{cpsub}
Let hypothesis \eqref{hp} be satisfied. Let $u$ be a solution to problem \eqref{problema} for some time $\tau=\tau_1>0$ and $\underline{u}$ a subsolution to problem \eqref{problema} for some time $\tau=\tau_2>0$. Suppose also that
$$
\operatorname{supp }\underline{u}|_{\R^N\times[0,S]} \text{ is compact for every }  \, S\in (0, \tau_2)\,.
$$
Then
\begin{equation}\label{eq173}
u\ge\underline{u} \quad \text{ in }\,\, \R^N \times \left(0,\min\{\tau_1,\tau_2\}\right)\,.
\end{equation}
\end{proposition}


In what follows we also consider solutions of equations of the form
\begin{equation}\label{eq189}
u_t = \frac 1{\rho(x)}\Delta(u^m) + u^p \quad \textrm{in }\,\, \Omega\times (0, \tau),
\end{equation}
where $\Omega\subseteq\mathbb R^N$. Solutions are meant in the following sense.

 \begin{definition}\label{soldom}
Let $\tau>0$, $p>1, m>1$. We say that a nonnegative function $u\in L^{\infty}(\Omega\times (0,S))$ for any $S<\tau$ is a solution of problem \eqref{problemalocale} if
\begin{equation}
\begin{aligned}
-\int_{\Omega}^{}\int_{0}^{\tau} \rho(x) u\, \varphi_t \,dt\,dx &= \int_{\Omega}^{}\int_{0}^{\tau}  u^m \Delta \varphi \,dt\,dx \\ &+ \int_{\Omega}^{}\int_{0}^{\tau} \rho(x) u^p \varphi \,dt\,dx
\end{aligned}
\end{equation}
for any $\varphi \in C_c^{\infty}(\overline{\Omega} \times [0,\tau))$ with $\varphi| _{\partial \Omega}=0$ for all $t\in [0,\tau)$. Moreover, we say that a nonnegative function $u\in L^{\infty}(\Omega\times (0,S))$ for any $S<\tau$ is a subsolution (supersolution) if it satisfies \eqref{veryweaklocale} with the inequality $"\le"$ ($"\ge"$) instead of $"="$, with $\varphi\geq 0$.
\label{soluzioneveryweaklocale}
\end{definition}

Finally, let us recall the following well-known criterion, that will be used in the sequel.
Let $\Omega\subseteq \mathbb R^N$ be an open set. Suppose that $\Omega=\Omega_1\cup \Omega_2$ with  $\Omega_1\cap \Omega_2=\emptyset$, and that   $\Sigma:=\partial \Omega_1\cap\partial \Omega_2$ is of class $C^1$. Let $n$ be the unit outwards normal to $\Omega_1$ at $\Sigma$.
Let
\begin{equation}\label{eq188}
u=\begin{cases}
u_1 & \textrm{in }\, \Omega_1\times [0, T),\\
u_2 & \textrm{in }\, \Omega_2\times [0, T)\,,
\end{cases}
\end{equation}
where $\partial_t u\in C(\Omega_1\times (0, T)), u_1^m\in C^2(\Omega_1\times (0, T))\cap C^1(\overline{\Omega}_1\times (0, T)) , \partial_t u_2\in C(\Omega_2\times (0, T))),\, u_2^m\in C^2(\Omega_2\times (0, T))\cap C^1(\overline{\Omega}_2\times (0, T)).$

\begin{lemma}\label{lemext}
Let assumption \eqref{hp} be satisfied.

(i) Suppose that
\begin{equation}\label{eq185}
\begin{aligned}
&\partial_t u_1 \geq \frac 1{\rho}\Delta u_1^m +u_1^p \quad \textrm{for any}\,\,\, (x,t)\in \Omega_1\times (0, T),\\
&\partial_t u_2 \geq  \frac 1{\rho}\Delta u_2^m + u_2^p \quad \textrm{for any}\,\,\, (x,t)\in \Omega_2\times (0, T),
\end{aligned}
\end{equation}
\begin{equation}\label{eq186}
u_1=u_2, \quad \frac{\partial u_1^m}{\partial n}\geq \frac{\partial u_2^m}{\partial n}\quad \textrm{for any }\,\, (x,t)\in \Sigma\times (0, T)\,.
\end{equation}
Then $u$, defined in \eqref{eq188}, is a supersolution to equation \eqref{eq189}, in the sense of Definition \ref{soldom}.

(ii)  Suppose that
\begin{equation}\label{eq185b}
\begin{aligned}
&\partial_t u_1 \leq  \frac 1{\rho}\Delta u_1^m +u_1^p \quad \textrm{for any}\,\,\, (x,t)\in \Omega_1\times (0, T),\\
&\partial_t u_2 \leq  \frac 1{\rho}\Delta u_2^m + u_2^p \quad \textrm{for any}\,\,\, (x,t)\in \Omega_2\times (0, T),
\end{aligned}
\end{equation}
\begin{equation}\label{eq186b}
u_1=u_2, \quad\frac{\partial u_1^m}{\partial n}\leq \frac{\partial u_2^m}{\partial n}\quad \textrm{for any}\,\, (x,t)\in \Sigma\times (0, T)\,.
\end{equation}
Then $u$, defined in \eqref{eq188}, is a subsolution to equation \eqref{eq189}, in the sense of Defi\-nition \ref{soldom}.
\end{lemma}

%

\section{Global existence: proofs}

In what follows we set $r\equiv |x|$. We construct a suitable family of supersolutions of equation
\begin{equation}
u_t =\frac{1}{\rho(x)}\Delta(u^m)+u^p \quad \text{ in } \R^N\times(0,+\infty).
\label{equazionecritical}
\end{equation}

\subsection{Order of decaying: $q=2$}
We assume \eqref{hp}, \eqref{hpSup} with $q=2$ and \eqref{hpC}. In order to construct a suitable family of supersolutions of \eqref{equazionecritical}, we define, for all $(x,t)\in \R^N \times (0,+\infty)$,
\begin{equation}
{\bar{u}}(x,t)\equiv \bar{u}(r(x),t):=C\zeta(t)\left [1-\frac{\log(r+r_0)}{a}\eta(t)\right]_{+}^{\frac{1}{m-1}}\,,
\label{subsupercritical}
\end{equation}
where $\eta$, $\zeta \in C^1([0, +\infty); [0, +\infty))$ and $C > 0$, $a > 0$, $r_0>e$.

\smallskip

Now, we compute
$$
\bar{u}_t - \frac{1}{\rho}\Delta(\bar{u}^m)-\bar{u}^p.
$$
To this aim, let us set
$$
F(r,t):= 1-\frac{\log(r+r_0)}{a}\,\eta(t)\,,
$$
and define
$$
\textit{D}_1:=\left \{ (x,t) \in [\R^N \setminus \{0\}]  \times (0,+\infty)\,\, | \,\,0<F(r,t)<1 \right \}.
$$
For any $(x,t) \in \textit{D}_1$, we have:

\begin{equation}
\begin{aligned}
\bar{u}_t &=C\zeta ' F^{\frac{1}{m-1}} + C\zeta \frac{1}{m-1} F^{\frac{1}{m-1}-1} \left ( -\frac{\log(r+r_0)}{a} \eta ' \right )\\
&=C\zeta ' F^{\frac{1}{m-1}} + C\zeta \frac{1}{m-1} \left (1-\frac{\log(r+r_0)}{a} \eta \right ) \frac{\eta'}{\eta}F^{\frac{1}{m-1}-1} - C\zeta \frac{1}{m-1} \frac{\eta'}{\eta} F^{\frac{1}{m-1}-1}\\
&=C\zeta ' F^{\frac{1}{m-1}} + C\zeta \frac{1}{m-1} \frac{\eta'}{\eta}F^{\frac{1}{m-1}} - C\zeta \frac{1}{m-1} \frac{\eta'}{\eta} F^{\frac{1}{m-1}-1}. \\
\end{aligned}
\label{dertempocritical}
\end{equation}
\begin{equation}
(\bar{u}^m)_r=-\frac{C^m}{a} \zeta^m \frac{m}{m-1} F^{\frac{1}{m-1}} \frac{1}{(r+r_0)}\eta.
\label{derprimacritical}
\end{equation}
\begin{equation}
\begin{aligned}
(\bar{u}^m)_{rr}&=-\frac{C^m}{a}\zeta^m \frac{m}{(m-1)^2} F^{\frac{1}{m-1}-1} \left(1-\frac{\log(r+r_0)}{a} \eta\right ) \eta \frac{1}{(r+r_0)^2\log(r+r_0)} \\ &+ \frac{C^m}{a} \zeta^m \frac{m}{(m-1)^2} F^{\frac{1}{m-1}-1} \frac{\eta}{(r+r_0)^2\log(r+r_0)} + \frac{C^m}{a} \zeta^m \frac{m}{m-1} F^{\frac{1}{m-1}} \frac{1}{(r+r_0)^2}\eta\\
&=-\frac{C^m}{a} \zeta^m \frac{m}{(m-1)^2} F^{\frac{1}{m-1}} \eta \frac{1}{(r+r_0)^2\log(r+r_0)}\\& + \frac{C^m}{a} \zeta^m \frac{m}{(m-1)^2} F^{\frac{1}{m-1}-1} \frac{\eta}{(r+r_0)^2\log(r+r_0)} + \frac{C^m}{a} \zeta^m \frac{m}{m-1} F^{\frac{1}{m-1}} \frac{1}{(r+r_0)^2}\eta.
\end{aligned}
\label{dersecondacritical}
\end{equation}
\begin{equation}\label{laplaciano}
\begin{aligned}
\Delta(\bar{u}^m)
&= \frac{(N-1)}{r}(\bar{u}^m)_r + (\bar{u}^m)_{rr} \\
&= \frac{(N-1)}{r}\left (-\frac{C^m}{a} \zeta^m \frac{m}{m-1} F^{\frac{1}{m-1}} \frac{1}{(r+r_0)}\eta \right ) \\
&-\frac{C^m}{a} \zeta^m \frac{m}{(m-1)^2} F^{\frac{1}{m-1}} \eta \frac{1}{(r+r_0)^2\log(r+r_0)} \\
&+ \frac{C^m}{a} \zeta^m \frac{m}{(m-1)^2} F^{\frac{1}{m-1}-1} \frac{\eta}{(r+r_0)^2\log(r+r_0)} \\
&+ \frac{C^m}{a} \zeta^m \frac{m}{m-1} F^{\frac{1}{m-1}} \frac{1}{(r+r_0)^2}\eta\\
&\leq  \frac{N-1}{r+r_0}\left (-\frac{C^m}{a} \zeta^m \frac{m}{m-1} F^{\frac{1}{m-1}} \frac{1}{(r+r_0)}\eta \right ) \\
&-\frac{C^m}{a} \zeta^m \frac{m}{(m-1)^2} F^{\frac{1}{m-1}} \eta \frac{1}{(r+r_0)^2\log(r+r_0)} \\
&+ \frac{C^m}{a} \zeta^m \frac{m}{(m-1)^2} F^{\frac{1}{m-1}-1} \frac{\eta}{(r+r_0)^2\log(r+r_0)} \\
&+ \frac{C^m}{a} \zeta^m \frac{m}{m-1} F^{\frac{1}{m-1}} \frac{1}{(r+r_0)^2}\eta\\
\end{aligned}
\end{equation}

We also define
\begin{equation}\label{coeff}
\begin{aligned}
&K:=\left [ \left (\frac{m-1}{p+m-2}\right)^{\frac{m-1}{p-1}} - \left (\frac{m-1}{p+m-2}\right)^{\frac{p+m-2}{p-1}} \right]\,>0\,,\\
&\bar{\sigma}(t) := \zeta ' + \zeta \frac{1}{m-1} \frac{\eta'}{\eta} + \frac{C^{m-1}}{a} \zeta^m \frac{m}{m-1} \eta k_1 \left (N-2\right ),\\
&\bar{\delta}(t) := \zeta \frac{1}{m-1} \frac{\eta'}{\eta} + C^{m-1} \zeta^m \frac{m}{(m-1)^2} \frac{\eta}{a}\frac 1 {\log(r_0)} k_2\,, \\
& \bar{\gamma}(t):=C^{p-1}\zeta^p\,.
\end{aligned}
\end{equation}

\begin{proposition}
Let $\zeta=\zeta(t)$, $\eta=\eta(t) \in C^1([0,+\infty);[0, +\infty))$. Let $K$, $\bar\sigma$, $\bar\delta$, $\bar\gamma$ be as defined in \eqref{coeff}. Assume \eqref{hp}, \eqref{hpSup} with $q=2$, \eqref{hpC} and that, for all $t\in (0,+\infty)$,
\begin{equation}
-\frac{\eta'}{\eta^2} \ge \frac{1}{\log(r_0)}\frac{C^{m-1}}{a} \zeta^{m-1}\frac{m}{m-1}k_2
\label{cond1supercritical}
\end{equation}
and
\begin{equation}
\zeta' + \frac{C^{m-1}}{a} \zeta^m \frac{m}{m-1} \eta \left[(N-2)k_1-\frac{k_2}{(m-1)\log(r_0)}\right] -C^{p-1}\zeta^p \ge 0\,.
\label{cond2supercritical}
\end{equation}
then $\bar{u}$ defined in \eqref{subsupercritical} is a supersolution of equation \eqref{equazionecritical}.
\label{propsupersolutioncritical}
\end{proposition}

\begin{proof}[Proof of Proposition \ref{propsupersolutioncritical}]

In view of \eqref{dertempocritical}, \eqref{derprimacritical}, \eqref{dersecondacritical} and \eqref{laplaciano}, for any $(x,t)\in D_1$,
\begin{equation}\label{eq1}
\begin{aligned}
\bar{u}_t - &\frac{1}{\rho}\Delta(\bar{u}^m)-\bar{u}^p \\
\geq \,\,
&C\zeta ' F^{\frac{1}{m-1}} + C\zeta \frac{1}{m-1} \frac{\eta'}{\eta}F^{\frac{1}{m-1}} - C\zeta \frac{1}{m-1} \frac{\eta'}{\eta} F^{\frac{1}{m-1}-1}\\ &+ \frac{C^m}{a} \zeta^m \frac{m}{m-1} \eta \frac{1}{\rho (r+r_0)^2} F^{\frac{1}{m-1}} \left (\frac{1}{(m-1)\log(r+r_0)}+N-2 \right) \\ & -\frac{C^m}{a} \zeta^m \frac{m}{(m-1)^2}F^{\frac{1}{m-1}-1}\frac{\eta}{\log(r+r_0)}\frac{1}{\rho (r+r_0)^2}- C^p \zeta^p F^{\frac{p}{m-1}}.
\end{aligned}
\end{equation}
Thanks to hypothesis \eqref{hp}, \eqref{hpSup} and \eqref{hpC}, we have
\begin{equation}\label{eq2}
\frac{1}{\log(r+r_0)} \ge 0\,, \,\,\,\,\,\,\, -\frac{1}{\log(r+r_0)} \ge -\frac{1}{\log(r_0)} \quad \text{for all }\,\, x\in \R^N\,,
\end{equation}
\begin{equation}\label{eq3}
\frac{1}{\rho (r+r_0)^2} \ge k_1\,, \,\,\,\,\,\,\,\,
-\frac{1}{\rho (r+r_0)^2} \ge -k_2 \quad \text{for all }\,\, x\in \R^N\,.
\end{equation}
From \eqref{eq1}, \eqref{eq2} and \eqref{eq3} we get,
\begin{equation}\label{eq5}
\begin{aligned}
&\bar{u}_t - \frac{1}{\rho}\Delta(\bar{u}^m)-\bar{u}^p  \\
&\ge CF^{\frac{1}{m-1}-1} \left \{F\left [\zeta ' + \zeta \frac{1}{m-1} \frac{\eta'}{\eta} + \frac{C^{m-1}}{a} \zeta^m \frac{m}{m-1} \eta (N-2) k_1 \right ] \right . \\
& \left .-\zeta \frac{1}{m-1} \frac{\eta'}{\eta} - \frac{C^{m-1}}{a} \zeta^m \frac{m}{(m-1)^2} \frac{1}{\log(r_0)}\eta k_2 - C^{p-1} \zeta^p F^{\frac{p+m-2}{m-1}} \right \}
\end{aligned}
\end{equation}
From \eqref{eq5} and \eqref{coeff}, we have
\begin{equation}\label{eq6}
\bar u_t - \frac{1}{\rho}\Delta(\bar u^m)-\bar u^p\geq C F^{\frac{1}{m-1}-1} \left[\bar{\sigma}(t)F - \bar{\delta}(t) - \bar{\gamma}(t)F^{\frac{p+m-2}{m-1}}\right ]\,.
\end{equation}
For each $t>0$, set
$$\varphi(F):=\bar{\sigma}(t)F - \bar{\delta}(t) - \bar{\gamma}(t)F^{\frac{p+m-2}{m-1}}, \quad F\in (0,1)\,.$$
Now our goal is to find suitable $C,a,\zeta, \eta$ such that, for each $t>0$,
$$\varphi(F) \ge 0 \quad \text{for any }\, F\in (0,1)\,.$$
We observe that $\varphi(F)$ is concave in the variable $F$. Hence it is sufficient to have $\varphi(F)$ positive in the extrema of the interval $(0,1)$. This reduces, for any $t>0$, to the conditions
\begin{equation}
\begin{aligned}
&\varphi(0) \ge 0\,,  \\
&\varphi(1) \ge 0  \,.
\end{aligned}
\end{equation}
These are equivalent to
$$
-\bar\delta(t) \ge 0\,, \quad \bar\sigma(t)-\bar\delta(t)-\bar\gamma(t) \ge 0\,,
$$
that is
$$
\begin{aligned}
&-\frac{\eta'}{\eta^2} \ge \frac{C^{m-1}}{a} \zeta^{m-1} \frac{m}{m-1}\frac{1}{\log(r_0)} k_2\,, \\
&\zeta' + \frac{C^{m-1}}{a} \zeta^m \frac{m}{m-1} \eta \left[ \left(N-2\right) k_1 -\frac{k_2}{(m-1)\log(r_0)}\right] -C^{p-1}\zeta^p \ge 0\,,
\end{aligned}
$$
which are guaranteed by \eqref{hpC}, \eqref{cond1supercritical} and \eqref{cond2supercritical}. Hence we have proved that
$$
\bar{u}_t - \frac{1}{\rho}\Delta(\bar{u}^m)-\bar{u}^p \ge 0 \quad \text{in } \textit{D}_1\,.
$$
Now observe that
\begin{itemize}
\item[]$\bar{u} \in C(\R^N \times [0,+\infty))$\,,
\item[]$\bar{u}^m \in C^1([\R^N \setminus \{0\}] \times [0,+\infty))$\,, and by the definition of $\bar{u}$\,,
\item[]$\bar{u}\equiv0$ in $[\R^N\setminus \textit{D}_1] \times [0,+\infty))$\,.
\end{itemize}
Hence, by Lemma \ref{lemext} (applied with $\Omega_1=D_1$, $\Omega_2=\R^N\setminus D_1$, $u_1=\bar u$, $u_2=0$, $u=\bar u$), $\bar u$ is a supersolution of equation
$$
\bar{u}_t - \frac{1}{\rho}\Delta(\bar{u}^m)-\bar{u}^p = 0 \quad \text{in } (\R^N \setminus \{0\})\times (0,+\infty)
$$
in the sense of Definition \ref{soldom}. Thanks to a Kato-type inequality,  since $\bar u^m_r(0,t)\le 0$, we can easily infer that $\bar u$ is a supersolution of equation \eqref{equazionecritical}
 in the sense of Definition \ref{soldom}.
\end{proof}

\begin{remark}\label{thmSuperC}
Let $$p>m$$ and assumption \eqref{hpC} be satisfied. Let $\omega:=\frac{C^{m-1}}{a}$. In Theorem \ref{teosupersolutioncritical} the precise hypotheses on parameters $C>0$, $\omega>0$, $T>0$ are the following:
\begin{equation}\label{eq35}
\frac{p-m}{p-1} \ge \omega \frac{m}{m-1} k_2 \frac{1}{\log(r_0)},
\end{equation}
\begin{equation}\label{eq36}
\omega \frac{m}{m-1}\left [ k_1(N-2)-\frac{k_2}{(m-1)\log(r_0)}\right ] \ge C^{p-1} + \frac{1}{p-1}\,.
\end{equation}
\end{remark}

\begin{lemma}\label{lemmaCsuper}
All the conditions in Remark \ref{thmSuperC} can be satisfied simultaneously.
\end{lemma}
\begin{proof}
Since $p>m$ the left-hand-side of \eqref{eq35} is positive. By \eqref{hpC}, we can select $\omega>0$ so that \eqref{eq35} holds and
$$
\omega \frac{m}{m-1}\left [ k_1(N-2)-\frac{k_2}{(m-1)\log(r_0)}\right ] > \frac{1}{p-1}\,.
$$
Then we take $C>0$ so small that \eqref{eq36} holds (and so $a>0$ is accordingly fixed).
\end{proof}

\begin{proof}[Proof of Theorem \ref{teosupersolutioncritical}]
We prove Theorem \ref{teosupersolutioncritical} by means of Proposition \ref{propsupersolutioncritical}. In view of Lemma \ref{lemmaCsuper}, we can assume that alla conditions in Remark \ref{thmSuperC} are fulfilled.
Set
$$
\zeta=(T+t)^{-\alpha}\,, \quad \eta=(T+t)^{-\beta}, \quad \text{for all } \quad t>0\,.
$$
Consider conditions \eqref{cond1supercritical}, \eqref{cond2supercritical} of Proposition \ref{propsupersolutioncritical} with this choice of $\zeta(t)$ and $\eta(t)$. Therefore we obtain
\begin{equation}\label{eq11}
\beta - \frac{C^{m-1}}{a}\frac{m}{m-1}k_2 (T+t)^{-\alpha(m-1)-\beta+1} \ge 0
\end{equation}
and
\begin{equation}\label{eq12}
\begin{aligned}
-\alpha (T+t)^{-\alpha-1}&+\frac{C^{m-1}}{a}\frac{m}{m-1}\left [k_1 (N-2) - \frac{k_2}{(m-1)\log(r_0)} \right ] (T+t)^{-\alpha m-\beta}\\ &-C^{p-1}(T+t)^{-\alpha p} \ge 0\,.
\end{aligned}
\end{equation}
We take
\begin{equation}\label{alphabeta}
\alpha=\frac{1}{p-1}\,, \quad \beta=\frac{p-m}{p-1}\,.
\end{equation}
Due to \eqref{alphabeta}, \eqref{eq11} and \eqref{eq12} become
\begin{equation}\label{eq14}
\frac{p-m}{p-1} \ge \frac{C^{m-1}}{a}\frac{m}{m-1}\frac{k_2}{\log(r_0)}\,,
\end{equation}
\begin{equation}\label{eq13}
\frac{C^{m-1}}{a}\frac{m}{m-1}\left[k_1(N-2)- \frac{k_2}{(m-1)\log(r_0)} \right ] \ge C^{p-1} + \frac{1}{p-1}\,.
\end{equation}
Therefore, \eqref{cond1supercritical} and \eqref{cond2supercritical} follow from assumptions \eqref{eq35} and \eqref{eq36}. Thus the conclusion follows by Propositions \ref{propsupersolutioncritical} and \ref{cpsup}.
\end{proof}

\subsection{Order of decaying: $q>2$}
We assume \eqref{hp}, \eqref{hpSup} and \eqref{hpS} for $q>2$ and \eqref{eq115}. In order to construct a suitable family of supersolutions of \eqref{equazionecritical}, we define, for all $(x,t)\in \R^N \times (0,+\infty)$,
\begin{equation}
{\bar{u}}(x,t)\equiv \bar{u}(r(x),t):=C\zeta(t)(r+r_0)^{-\frac{\bar b}{m}};
\label{eq110}
\end{equation}
where $\zeta \in C^1([0, +\infty); [0, +\infty))$ and $C > 0$, $r_0>0$.

\smallskip

Now, we compute
$$
\bar{u}_t - \frac{1}{\rho}\Delta(\bar{u}^m)-\bar{u}^p.
$$

For any $(x,t) \in \big[\R^N\setminus \{0\}\big]\times (0,+\infty)$, we have:

\begin{equation}
\bar{u}_t =C\,\zeta'\, (r+r_0)^{-\frac{b}{m}}\,.
\label{eq111}
\end{equation}
\begin{equation}
(\bar{u}^m)_r=-\,\bar b \,C^m \,\zeta^m \,(r+r_0)^{-\bar b -1}\,.
\label{eq112}
\end{equation}
\begin{equation}
(\bar{u}^m)_{rr}=\bar b\,(\bar b+1)\,C^m \,\zeta^m \,(r+r_0)^{-\bar b -2}\,.
\label{eq113}
\end{equation}

\begin{proposition}\label{propsupersolutionsuper}
Let $\zeta=\zeta(t) \in C^1[0,+\infty); [0, +\infty)), \zeta'\geq 0$. Assume \eqref{hp}, \eqref{hpSup} and \eqref{hpS} for $q>2$, \eqref{eq115}, and that
\begin{equation}
\bar bk_1(N-2-\bar b)C^m \zeta^m -\bar c \, C^p\zeta^p  > 0\,.
\label{eq114}
\end{equation}
Then $\bar{u}$ defined in \eqref{eq110} is a supersolution of equation \eqref{equazionecritical}.
\end{proposition}

\begin{proof}[Proof of Proposition \ref{propsupersolutionsuper}]
In view of \eqref{eq111}, \eqref{eq112}, \eqref{eq113} and the fact that
\[\frac1{(r+r_0)^{\bar b+1}r}\geq \frac1{(r+r_0)^{\bar b+2}}\quad \textrm{for any}\,\, x\in \mathbb R^N,\]
we get, for any $(x,t)\in (\R^N\setminus \{0\})\times (0,+\infty)$,
\begin{equation}\label{eq120}
\begin{aligned}
&\bar{u}_t-\frac 1 {\rho} \Delta(\bar{u}^m) -\bar{u}^p\\
&\geq C\zeta'(r+r_0)^{-\frac{\bar b}{ m}} + \frac 1 {\rho} \left \{(N-2-\bar b)C^m \zeta^m \bar b (r+r_0)^{-\bar b-2}\right \} - C^p \zeta^p (r+r_0)^{-\frac{\bar bp}{m}}.
\end{aligned}
\end{equation}
Thanks to hypothesis \eqref{hpSup}, \eqref{hpS} and \eqref{eq115}, we have
\begin{equation}\label{eq121}
\begin{aligned}
&\frac{(r+r_0)^{-\bar b-2}}{\rho} \ge k_1(r+r_0)^{-\bar b-2+q}=k_1\,,\\
&-(r+r_0)^{-\frac{\bar bp}{m}} \ge - \bar c
\end{aligned}
\end{equation}
Since $\zeta'\geq 0$, from \eqref{eq120} and \eqref{eq121} we get
\begin{equation}\label{eq122}
\bar{u}_t-\frac \ {\rho} \Delta(\bar{u}^m) -\bar{u}^p\ge k_1 \bar b (N-2-\bar b)C^m\zeta^m - \bar{c} \,C^p \zeta^p\,.
\end{equation}
Hence we get the condition
\begin{equation}\label{eq123}
k_1 \bar b (N-2-\bar b)C^m\zeta^m - \bar{c} \,C^p \zeta^p \ge 0\,,
\end{equation}
which is guaranteed by \eqref{hpS} and \eqref{eq114}. Hence we have proved that
$$
\bar{u}_t-\frac 1 {\rho} \Delta(\bar{u}^m) -\bar{u}^p \ge 0 \quad \text{in }\,\, (\R^N\setminus \{0\})\times (0, +\infty)\,.
$$
Now observe that
$$
\begin{aligned}
&\bar u \in C(\R^N\times [0,+\infty))\,,\\
&\bar u^m \in C^1([\R^N \setminus \{0\}]\times [0,+\infty))\,, \\
& \bar u_r^m(0,t)\le 0\,.
\end{aligned}
$$
Hence, thanks to a Kato-type inequality we can infer that $\bar u$ is a supersolution to equation \eqref{equazionecritical}
in the sense of Definition \ref{soldom}.

\end{proof}

\begin{remark}\label{thmSuperS}
Let $$q>2$$ and assumption \eqref{hpS} be satisfied. In Theorem \ref{teosupersolutionsuper} the precise hypotheses on parameters $\alpha$, $C>0$, $T>0$ are as follows.
\begin{itemize}
\item[(a)] Let $p<m$. We require that
\begin{equation}\label{eq125}
\alpha>0,
\end{equation}
\begin{equation}\label{eq126}
\bar b\,k_1(N-2-\bar b)C^m\,- \bar c \,C^p \, \ge 0
\end{equation}
\item[(b)] Let $p>m$. We require that
\begin{equation}\label{eq127}
\alpha=0,
\end{equation}
\begin{equation}\label{eq128}
\bar b\,k_1(N-2-\bar b)C^m - \bar c \,C^p\ge 0
\end{equation}
\end{itemize}
\end{remark}

\begin{lemma}\label{lemmaSsuper}
All the conditions in Remark \ref{thmSuperS} can hold simultaneously.
\end{lemma}
\begin{proof}
(a) We observe that, due to \eqref{hpS}, $$N-2-\bar b>0.$$ Therefore, we can select $C>0$ sufficiently large to guarantee \eqref{eq126}.\newline
(b) We choose $C>0$ sufficiently small to guarantee \eqref{eq128}.
\end{proof}

\begin{proof}[Proof of Theorem \ref{teosupersolutionsuper}]
We now prove Theorem \ref{teosupersolutionsuper} in view of Proposition \ref{propsupersolutionsuper}. In view of Lemma \ref{lemmaSsuper} we can assume that all conditions in Remark \ref{thmSuperS} are fulfilled.
Set
$$
\zeta(t)=(T+t)^{\alpha}, \quad \text{for all} \quad t \ge 0\,.
$$

Let $p<m$. Inequality \eqref{eq114} reads
\[\bar b\,k_1(N-2-\bar b)C^m(T+t)^{m\alpha}\,- \bar c \,C^p(T+t)^{p\alpha} \, \ge 0\quad \textrm{for all }\,\, t>0\,.\]
This follows from \eqref{eq125} and \eqref{eq126}, for $T>1$. Hence, by Propositions \ref{propsupersolutionsuper} and \ref{prop1} the thesis follows in this case.

Let $p>m$. Conditions \eqref{eq127} and \eqref{eq128} are equivalent to \eqref{eq114}. Hence, by Propositions \ref{propsupersolutionsuper} and \ref{prop1} the thesis follows in this case too. The proof is complete.

\end{proof}

\section{Blow-up: proofs}\label{proofsectionB}

In what follows we set $r\equiv |x|$. We construct a suitable family of subsolutions of equation
\begin{equation}\label{equazionecritical2}
u_t =\frac{1}{\rho(x)}\Delta(u^m)+u^p \quad \text{ in } \R^N\times(0, T).
\end{equation}

\subsection{Order of decaying: $q=2$}

Suppose \eqref{hp}, \eqref{hpSub} and \eqref{hpCsub}. To construct a suitable family of subsolution of \eqref{equazionecritical2}, we define, for all $(x,t)\in [\R^N \setminus B_{e}] \times (0,T)$,
\begin{equation}
{\underline{u}}(x,t)\equiv \underline{u}(r(x),t):=C\zeta(t)\left [1-\frac{\log(r)}{a}\eta(t)\right]_{+}^{\frac{1}{m-1}},
\label{subcritical}
\end{equation}
and
\begin{equation}\label{w}
w(x,t)\equiv w(r(x),t) :=
\begin{cases}
\underline u(x,t) \quad \text{in } [\R^N \setminus B_{e}] \times (0,T), \\
v(x,t) \quad \text{in } B_{e} \times (0,T),
\end{cases}
\end{equation}
where
\begin{equation}
v(x,t) \equiv v(r(x),t):= C\zeta(t) \left [ 1-\frac{r^2+e^2}{2e^2} \frac{\eta}{a} \right ]^{\frac{1}{m-1}}_{+}\,.
\label{vcritical}
\end{equation}
Let us set
$$
F(r,t):= 1-\frac{\log(r)}{a}\eta(t)\,,
$$
and
$$
G(r,t):= 1-\frac{r^2+e^2}{2e^2}\frac{\eta(t)}{a}\,.
$$
For any $(x,t) \in (\R^N\setminus B_e) \times (0,T)$, we have:
\begin{equation}
\begin{aligned}
\underline{u}_t &=C\zeta ' F^{\frac{1}{m-1}} + C\zeta \frac{1}{m-1} F^{\frac{1}{m-1}-1} \left ( -\frac{\log(r)}{a} \eta ' \right )= \\
&=C\zeta ' F^{\frac{1}{m-1}} + C\zeta \frac{1}{m-1} \left (1-\frac{\log(r)}{a} \eta \right ) \frac{\eta'}{\eta}F^{\frac{1}{m-1}-1} - C\zeta \frac{1}{m-1} \frac{\eta'}{\eta} F^{\frac{1}{m-1}-1}= \\
&=C\zeta ' F^{\frac{1}{m-1}} + C\zeta \frac{1}{m-1} \frac{\eta'}{\eta}F^{\frac{1}{m-1}} - C\zeta \frac{1}{m-1} \frac{\eta'}{\eta} F^{\frac{1}{m-1}-1}. \\
\end{aligned}
\label{dertempocriticalsub}
\end{equation}
\begin{equation}
(\underline{u}^m)_r=-\frac{C^m}{a} \zeta^m \frac{m}{m-1} F^{\frac{1}{m-1}} \frac{1}{r}\eta.
\label{derprimacriticalsub}
\end{equation}
\begin{equation}
\begin{aligned}
(\underline{u}^m)_{rr}&=-C^m \zeta^m \frac{m}{(m-1)^2} F^{\frac{1}{m-1}-1} \left(1-\frac{\log(r)}{a} \eta\right ) \eta \frac{1}{r^2\log(r)} \\ &+ \frac{C^m}{a} \zeta^m \frac{m}{(m-1)^2} F^{\frac{1}{m-1}-1} \frac{\eta}{r^2 \log(r)} + \frac{C^m}{a} \zeta^m \frac{m}{m-1} F^{\frac{1}{m-1}} \frac{1}{r^2}\eta=\\
&=-\frac{C^m}{a} \zeta^m \frac{m}{(m-1)^2} F^{\frac{1}{m-1}} \eta \frac{1}{r^2\log(r)}\\& + \frac{C^m}{a} \zeta^m \frac{m}{(m-1)^2} F^{\frac{1}{m-1}-1} \frac{\eta}{r^2 \log(r)} + \frac{C^m}{a} \zeta^m \frac{m}{m-1} F^{\frac{1}{m-1}} \frac{1}{r^2}\eta.
\end{aligned}
\label{dersecondacriticalsub}
\end{equation}
For any $(x,t) \in B_e \times (0,T)$, we have:
\begin{equation}
\begin{aligned}
v_t &=C\zeta ' G^{\frac{1}{m-1}} + C\zeta \frac{1}{m-1} G^{\frac{1}{m-1}-1} \left ( -\frac{r^2+e^2}{2e^2} \frac{\eta'}{a} \right )= \\
&=C\zeta ' G^{\frac{1}{m-1}} + C\frac{\zeta}{m-1} \left (1-\frac{r^2+e^2}{2e^2} \frac{\eta}{a} \right ) \frac{\eta'}{\eta}G^{\frac{1}{m-1}-1} - C\zeta \frac{1}{m-1} \frac{\eta'}{\eta} G^{\frac{1}{m-1}-1}= \\
&=C\zeta ' G^{\frac{1}{m-1}} + C\zeta \frac{1}{m-1} \frac{\eta'}{\eta}G^{\frac{1}{m-1}} - C\zeta \frac{1}{m-1} \frac{\eta'}{\eta} G^{\frac{1}{m-1}-1}\,.
\end{aligned}
\label{A}
\end{equation}
\begin{equation}
(v^m)_r=-\frac{C^m}{a} \zeta^m \frac{m}{m-1} G^{\frac{1}{m-1}} \frac{r}{e^2}\eta\,.
\label{B}
\end{equation}
\begin{equation}
(v^m)_{rr}=-C^m \zeta^m \frac{m}{m-1} G^{\frac{1}{m-1}} \frac 1{e^2} \frac {\eta}{a}+ \frac{C^m}{a^2} \zeta^m \frac{m}{(m-1)^2} G^{\frac{1}{m-1}-1} \eta^2 \frac{r^2}{e^4}\,.
\label{C}
\end{equation}
We also define
\begin{equation}
\begin{aligned}
& \underline\sigma(t) := \zeta ' + \zeta \frac{1}{m-1} \frac{\eta'}{\eta} + \frac{C^{m-1}}{a} \zeta^m \frac{m}{m-1} \eta k_2 \left (N-2+ \frac{1}{m-1} \right )\,,\\
&\underline \delta(t) := \zeta \frac{1}{m-1} \frac{\eta'}{\eta}\,,  \\
& \underline\gamma(t):=C^{p-1} \zeta^p, \\
& \underline \sigma_0(t) := \zeta'+\frac{\zeta}{m-1} \frac{\eta'}{\eta} + \frac{N}{e^2}\rho_2 \frac{C^{m-1}}{a} \zeta^{m} \frac{m}{m-1} \eta\, , \\
& K:= \left (\frac{m-1}{p+m-2}\right)^{\frac{m-1}{p-1}} - \left (\frac{m-1}{p+m-2}\right)^{\frac{p+m-2}{p-1}}>0.
\end{aligned}
\label{coeff2}
\end{equation}

\begin{proposition}
Let $p>m$. Let $T\in (0,\infty)$, $\zeta$, $\eta \in C^1([0,T); [0, T))$.
Let $\underline{\sigma}$, $\underline{\delta}$, $\underline{\gamma}$, $\underline{\sigma}_0$, $\textit{K}$ be defined in \eqref{coeff2}. Assume that, for all $t\in(0,T)$,
\begin{equation}\label{eq40}
\underline\sigma(t)>0, \quad K[\underline\sigma(t)]^{\frac{p+m-2}{p-1}} \le \underline\delta(t) \underline\gamma(t)^{\frac{m-1}{p-1}}\,,
\end{equation}
\begin{equation}\label{eq41}
(m-1)\underline \sigma(t) \le (p+m-2) \underline\gamma(t)\,.
\end{equation}
\begin{equation}\label{eq42}
\underline \sigma_0(t)>0, \quad K[\underline{\sigma}_0](t)^{\frac{p+m-2}{p-1}} \le \underline{\delta}(t) \underline{\gamma}(t)^{\frac{m-1}{p-1}}, \end{equation}
\begin{equation}\label{eq43}
(m-1) \underline{\sigma}_0(t) \le (p+m-2) \underline{\gamma}(t)\,.
\end{equation}
Then $w$ defined in \eqref{w} is a  subsolution of equation \eqref{equazionecritical2}.
\label{propsubsolutioncritical}
\end{proposition}

\begin{proof}[Proof of Proposition \ref{propsubsolutioncritical}]
Let $\underline{u}$ be as in \eqref{subcritical} and set
$$
\textit{D}_2:=\left \{ (x,t) \in (\R^N \setminus B_e) \times (0,T)\,\, |\,\, 0<F(r,t)<1 \right \}.
$$
In view of \eqref{dertempocriticalsub}, \eqref{derprimacriticalsub}, \eqref{dersecondacriticalsub}, we obtain, for all $(x,t)\in D_2$,
$$
\begin{aligned}
&\underline{u}_t - \frac{1}{\rho}\Delta(\underline {u}^m)- \underline{ u}^p\\
&=C\zeta ' F^{\frac{1}{m-1}} + C\zeta \frac{1}{m-1} \frac{\eta'}{\eta}F^{\frac{1}{m-1}} - C\zeta \frac{1}{m-1} \frac{\eta'}{\eta} F^{\frac{1}{m-1}-1}\\ & + F^{\frac{1}{m-1}} \frac{C^m}{a} \zeta^m \frac{m}{m-1} \eta \frac{1}{\rho r^2} \left ( \frac{1}{(m-1)\log(r)} +N-1 \right) - \frac{C^m}{a}\zeta^m \frac{m}{(m-1)^2} F^{\frac{1}{m-1}-1} \frac{\eta}{\log(r)} \frac{1}{\rho r^2}\\& -C^p \zeta^p F^{\frac{p}{m-1}}.
\end{aligned}
$$
In view of hypotheses \eqref{hpSub} and \eqref{hpCsub}, we can infer that
\begin{equation}\label{eq44}
\frac{1}{\rho r^2} \le k_2\,, \quad -\frac{1}{\rho r^2} \le -k_1\, \quad \text{for all }\,\,\ x\in \R^N\setminus B_e\,.
\end{equation}
Moreover,
\begin{equation}\label{eq45}
-1 \le -\frac{1}{\log(r)} \le 0\,, \quad \frac{1}{\log(r)} \le 1\,, \quad \text{for all }\,\,\ x\in \R^N\setminus B_e\,.
\end{equation}
From \eqref{eq44} and \eqref{eq45} we have
\begin{equation}\label{eq46}
\begin{aligned}
&\underline u_t - \frac{1}{\rho}\Delta(\underline u^m)- \underline u^p\\
&\le CF^{\frac{1}{m-1}-1} \left \{F\left [\zeta ' + \zeta\, \frac{1}{m-1}\, \frac{\eta'}{\eta} + \frac{C^{m-1}}{a}\, \zeta^m \,\frac{m}{m-1} \,\eta\, k_2 \right .\right.\\
&\,\,\,\,\,\left .\times \left(N-2+ \frac{1}{m-1} \right ) \right ] \left . -\zeta \frac{1}{m-1} \frac{\eta'}{\eta}  - C^{p-1} \zeta^p F^{\frac{p+m-2}{m-1}} \right \}.
\end{aligned}
\end{equation}
Thanks to \eqref{coeff2} and \eqref{eq46}
\begin{equation}\label{eqf1}
\underline u_t - \frac{1}{\rho}\Delta(\underline u^m)- \underline u^p \leq CF^{\frac{1}{m-1}-1}\varphi(F),
\end{equation}
where
\begin{equation}\label{eq47}
\varphi(F):=\underline\sigma(t)F - \underline\delta(t) - \underline\gamma(t)F^{\frac{p+m-2}{m-1}}.
\end{equation}
Due to \eqref{eqf1}, our goal is to find suitable $C>0$, $a>0$, $\zeta$, $\eta$ such that
$$
\varphi(F) \le 0\,, \quad \text{for all}\,\,  F \in (0,1)\,.
$$
To this aim, we impose that
$$
\sup_{F\in (0,1)}\varphi(F)=\max_{F\in (0,1)}\varphi(F)= \varphi (F_0)\leq 0\,,
$$
for some $F_0\in(0,1)$. We have
$$
\begin{aligned}
\frac{d \varphi}{dF}=0
&\iff \underline\sigma(t) - \frac{p+m-2}{m-1}\underline \gamma(t) F^{\frac{p-1}{m-1}} =0 \\ & \iff F_0= \left [\frac{m-1}{p+m-2} \frac{\underline\sigma(t)}{\underline\gamma(t)} \right ]^{\frac{m-1}{p-1}}\,.
\end{aligned}
$$
Then,
$$
\varphi(F_0)= K \dfrac{\underline\sigma(t)^{\frac{p+m-2}{p-1}}}{\underline\gamma(t)^{\frac{m-1}{p-1}}} - \underline\delta(t)
$$
where the coefficient $K=K(m,p)$ has been defined in \eqref{coeff2}.
By hypotheses \eqref{eq40} and \eqref{eq41}
\begin{equation}\label{eq50}
\varphi(F_0) \le 0\,, \quad 0<F_0 \le 1\,.
\end{equation}
So far, we have proved that
\begin{equation}\label{eq51}
\underline{u}_t-\frac{1}{\rho(x)}\Delta(\underline{u}^m)-\underline{u}^p \le 0 \quad \text{ in } D_2\,.
\end{equation}
Furthermore, since $\underline{u}^m\in C^1([\R^N\setminus B_e]\times[0,T))$, due to Lemma \ref{lemext} (applied with $\Omega_1=D_2, \Omega_2=\R^N\setminus[B_e\cup D_2], u_1=\underline u, u_2=0, u=\underline u$), it follows that $\underline u$ is a subsolution to equation
\begin{equation}\label{eq58}
\underline{u}_t-\frac{1}{\rho(x)}\Delta(\underline{u}^m)-\underline{u}^p = 0 \quad \text{ in } [\R^N \setminus B_e]\times (0,T)\,,
\end{equation}
in the sense of Definition \ref{soldom}.

Let
$$
D_3:= \left \{ (x,t) \in B_e \times (0,T) \,\,|\,\, 0<G<1\right \}\,.
$$
In view of \eqref{A}, \eqref{B} and \eqref{C}, for all $(x,t)\in D_3$,
\begin{equation}\label{eq56}
\begin{aligned}
&v_t-\frac{1}{\rho(x)}\Delta(v^m)-v^p\\
&=CG^{\frac 1{m-1}-1}\left \{G\left [\zeta'+\frac{\zeta}{m-1}\frac{\eta'}{\eta}+\frac 1 {\rho} \frac{C^{m-1}}{a} \zeta^m \frac m {m-1} \frac{N-1}{e^2} \eta \frac 1 {\rho} \frac{C^{m-1}}{a} \zeta^m \frac m {m-1} \frac{1}{e^2} \eta \right ]  \right .  \\
& + \left . - \frac{\zeta}{m-1}\frac{\eta'}{\eta} -  \frac 1 {\rho} \frac{C^{m-1}}{a^2} \zeta^m \frac m {(m-1)^2} \frac{r^2}{e^4} \eta^2 - C^{p-1}\zeta^p G^{\frac{p+m-2}{m-1}} \right \}
\end{aligned}
\end{equation}
Using \eqref{rho}, \eqref{eq56} yield, for all $(x,t)\in D_3$,
\begin{equation}\label{eq57}
\begin{aligned}
v_t - &\frac{1}{\rho}\Delta(v^m)-v^p \\
&\le CG^{\frac 1{m-1}-1}\left \{G\left [\zeta'+\frac{\zeta}{m-1}\frac{\eta'}{\eta}+\rho_2\frac{C^{m-1}}{a} \zeta^m \frac m {m-1} \frac{N}{e^2} \eta \right ] \right . \\& \left .- \frac{\zeta}{m-1}\frac{\eta'}{\eta} - C^{p-1}\zeta^p G^{\frac{p+m-2}{m-1}} \right \}\,.
\end{aligned}
\end{equation}
Thanks to \eqref{coeff} and \eqref{eq57},
\begin{equation}\label{eqf2}
v_t - \frac{1}{\rho}\Delta(v^m)-v^p \leq CG^{\frac 1{m-1}-1} \psi(G),
\end{equation}
where
\begin{equation}\label{eq58b}
\psi(G):=\underline\sigma_0(t)G - \underline\delta(t) - \underline\gamma(t)G^{\frac{p+m-2}{m-1}}.
\end{equation}
Now, by the same arguments used to obtain \eqref{eq58}, in view of \eqref{eq42} and \eqref{eq43} we can infer that
$$
\psi(G)\le 0\, \quad 0<G\le 1\,.
$$
So far, due to \eqref{eqf2}, we have proved that
\begin{equation}\label{eq59}
v_t-\frac{1}{\rho(x)}\Delta(v^m)-v^p \le 0 \quad \text{ for any }\,\, (x,t)\in D_3\,.
\end{equation}
Moreover, by Lemma \ref{lemext} $v$ is a subsolution of equation
\begin{equation}\label{eq60}
v_t-\frac{1}{\rho(x)}\Delta(v^m)-v^p = 0 \quad \text{ in }\,\, B_e\times(0,T)\,,
\end{equation}
in the sense of Definition \ref{soldom}. Now, observe that $w \in C(\R^N \times [0,T))$, indeed,
$$
\underline{u}  = v = C \zeta(t)\left [ 1- \frac{\eta(t)}{a} \right ]_+^{\frac{1}{m-1}} \quad \text{in}\,\, \partial B_e\times (0,T)\,.
$$
Moreover, $w^m \in C^1(\R^N \times [0,T))$, indeed,
$$
(\underline{u}^m)_r = (v^m)_r  = -C^m \zeta(t)^m \frac{m}{m-1} \frac{1}{e}\frac{\eta(t)}{a} \left [ 1- \frac{\eta(t)}{a} \right ]_+^{\frac{1}{m-1}} \quad \text{in}\,\, \partial B_e\times (0,T)\,.
$$
Hence, by Lemma \ref{lemext} again, $w$ is a subsolution to equation \eqref{equazionecritical2} in the sense of Definition \ref{soldom}.
\end{proof}

\begin{remark}\label{thmSubC}
Let $$p>m\,,$$ and assumptions \eqref{hpSub} and \eqref{hpCsub} be satisfied. Let define $\omega:= \frac{C^{m-1}}{a}$. In Theorem \ref{teosubsolutioncritical}, the precise hypotheses on parameters $C>0$, $a>0$, $\omega>0$ and $T>0$ is the following.
\begin{equation}\label{eq61}
\max \left \{1+mk_2\frac{C^{m-1}}{a} \left (N-2+\frac{1}{m-1}\right )\,; 1+m\rho_2\frac{C^{m-1}}{a} \frac{N}{e^2}\right \} \le (p+m-2)C^{p-1}\,,
\end{equation}
\begin{equation}\label{eq62}
\begin{aligned}
\dfrac {K}{(m-1)^{\frac{p+m-2}{p-1}}}\,\, \max& \left \{ \left [1+mk_2\frac{C^{m-1}}{a} \left (N-2+\frac{1}{m-1}\right ) \right ]^{\frac{p+m-2}{p-1}}\,\right . ;\\ &\left .\left ( 1+m\rho_2\frac{C^{m-1}}{a} \frac{N}{e^2} \right )^{\frac{p+m-2}{p-1}} \right \} \,\le \,\frac{p-m}{(m-1)(p-1)}C^{m-1} \,.
\end{aligned}
\end{equation}

\end{remark}

\begin{lemma}\label{lemmaCsub}
All the conditions in Remark \ref{thmSubC} can hold simultaneously.
\end{lemma}

\begin{proof}
We can take $\omega>0$ such that
$$
\omega_0\le \omega \le \omega_1
$$
for suitable $0<\omega_0<\omega_1$ and we can choose $C>0$ sufficiently large to guarantee \eqref{eq61} and \eqref{eq62} (so, $a>0$ is fixed, too).
\end{proof}

\begin{proof}[Proof of Theorem \ref{teosubsolutioncritical}]
We now prove Theorem \ref{teosubsolutioncritical}, by means of Proposition \ref{propsubsolutioncritical}. In view of Lemma \ref{lemmaCsub} we can assume that all conditions of Remark \ref{thmSubC} are fulfilled. Set
$$
\zeta=(T-t)^{-\alpha}\,, \quad \eta=(T-t)^{-\beta}\,, \quad \text{for all} \quad t>0\,,
$$
and $\alpha$ and $\beta$ as defined in \eqref{alphabeta}. Then
\begin{equation}
\begin{aligned}
& \underline\sigma(t) := \left [ \frac{1}{m-1}+\frac{C^{m-1}}{a}\frac{m}{m-1}k_2\left(\frac 1 {m-1}+N-2\right)\right] (T-t)^{-\frac{p}{p-1}}\,,\\
&\underline \delta(t) := \frac{p-m}{(m-1)(p-1)} (T-t)^{-\frac{p}{p-1}}\,,  \\
& \underline\gamma(t):=C^{p-1} (T-t)^{-\frac{p}{p-1}}\,, \\
& \underline \sigma_0(t) :=  \frac 1 {m-1} \left [1+\frac{\rho_2 N m}{e^2} \frac{C^{m-1}}{a} \right](T-t)^{-\frac{p}{p-1}} \,.
\end{aligned}
\label{sost}
\end{equation}

Let $p>m$. Condition \eqref{eq61} implies \eqref{eq40}, \eqref{eq41}, while condition \eqref{eq62} implies \eqref{eq42}, \eqref{eq43}. Hence by Propositions \ref{propsubsolutioncritical} and \ref{cpsub} the thesis follows.
\end{proof}

\end{document}